\documentclass[leqno,12pt]{amsart}
\usepackage{stmaryrd} 
\usepackage{amssymb}
\theoremstyle{plain} 
\newtheorem{theorem}{\indent\sc Theorem}[section]
\newtheorem{lemma}[theorem]{\indent\sc Lemma}

\theoremstyle{definition} 

\newtheorem{remark}[theorem]{\indent\sc Remark}

\begin{document}

\title[ Convective Cahn-Hilliard equation]{Global well-posedness and decay of solutions to the Cauchy problem of   convective Cahn-Hilliard equation }%

\author[X. Zhao]{Xiaopeng Zhao$^\dagger$
} 




\address{$^\dagger$
College of Sciences,
Northeastern University,
Shenyang 110819,,
P. R. China\endgraf
}
\email{zhaoxiaopeng@mail.neu.edu.cn}


\begin{abstract}
In this paper, we consider the global well-posedness and time-decay rates of solution to the Cauchy problem for 3D convective Cahn-Hilliard equation with double-well potential  via a refined pure energy method.  In particular, the optimal decay rates of the higher-order spatial derivatives of the solution are obtained, the $\dot{H}^{-s}$ ($0<s\leq\frac12$)   negative Sobolev norms is shown to be preserved along time evolution and enhance the decay rates.

\end{abstract}\maketitle{\small
\section{Introduction} \label{sect1}

The convective Cahn-Hilliard equation \cite{EJMP,Eur,Wa,Gol}
 \begin{equation} \label{1-0} \partial_tu+\Delta^2u=\Delta\varphi(u)+\vec{\beta}\cdot\nabla  \psi(u),
 \end{equation}
 arises
naturally as a continuous model for the formation of facets and
corners in crystal growth. In equation (\ref{1-0}), $u(x,t)$ denotes the
slope of the interface \cite{Gol}, the convective term
$\vec{\beta}\cdot\nabla  \psi(u)$  stems
from the effect of kinetic that provides an independent flux of the
order parameter, similar to the effect of an external field in
spinodal decomposition of a driven system \cite{Gol}, $\varphi(u)$ stands for the derivative of  a configuration potential $\Phi(u)=\int_0^u\varphi(s)ds$, respectively.   Usually, we take $\varphi(u)$ as the derivation of a double-well potential
\begin{equation}
\nonumber
\varphi(s)=\Phi'(s)=s(s^2-1),\quad \Phi(s)=\frac14(s^2-1)^2,
\end{equation}
or   a singular potential (see\cite{Fri,Gal})
$$
\varphi_{\log}(s)=-\kappa_0s+\kappa_1\ln\frac{1+s}{1-s},\quad0<\kappa_1<\kappa_0.$$
 For small driving force $\vec{\beta}\rightarrow0$, equation (\ref{1-0}) is   reduced to the well-known classical Cahn-Hilliard
equation \cite{CH1,CH2,CH3,CH4,CH5,Sch}.

 A large amount of literature has been produced about the convective Cahn-Hilliard equation  in a bounded domain, subject to suitable boundary conditions. For example,  Zaks et al. \cite{Zaks}
investigated the bifurcations of stations periodic solutions of a
convective Cahn-Hilliard equation; Eden and Kalantarov \cite{Eden1,Eden2} established
some results on the existence of a compact attractor for the
convective Cahn-Hilliard equation with periodic boundary conditions
in one space dimension and three space dimension; Della Porta and Grasselli \cite{Della} considered the initial-boundary value problem of convective nonlocal Cahn-Hilliard equation as dynamical systems and showed that they have bounded absorbing sets and global attractors; Zhao and Liu  \cite{ZhaoAA,Zhaox} investigated the existence of optimal solutions and optimality condition for  thei initial-boundary value problem of convective Cahn-Hilliard equation;  Rocca and Sprekels \cite{RS} studied a distributed control problem for a 3D convective nonlocal Cahn-Hilliard-type  system involving a degenerate mobility and a singular potential. In \cite{L1,LY}, Liu et. al. considered properties of solutions for the initial-boundary value problem of the convective Cahn-Hilliard equation with nonconstant mobility and degenerate mobility.  
\begin{remark}
The Cahn-Hilliard equation
\begin{equation} \label{CH}
       \partial_tu= \Delta[-\gamma\Delta u+\varphi(u)],
                          \end{equation}    was used to describe  phase transition problems in binary metallic alloys \cite{CH2}, the representation of the tumor growth process \cite{Al}, color image inpainting \cite{Col} and other phenomenons.
 The convective Cahn-Hilliard equation can be seen as a modification of equation (\ref{CH}).  

\end{remark}

The Cauchy problem of the convective Cahn-Hilliard   in $\mathbb{R}^N$ ($N\in\mathbb{Z}^+$) has the following form:
 \begin{equation} \label{1-1}
\left\{ \begin{aligned}
         &\partial_tu+\Delta^2u=\Delta\varphi(u)+\gamma\nabla\cdot\psi(u), \quad (x,t)\in\mathbb{R}^N\times(0,\infty),\\
                 &u(x,0)=u_0(x),
                          \end{aligned} \right.
                          \end{equation} where $\gamma>0$ is a positive constant. In \cite{AMPA}, assuming that the  initial data $u_0(x)$ satisfies  $u_0(x) \in L^{\frac{N(l-1)}3}(\mathbb{R}^N)\bigcap L^{\infty}(\mathbb{R}^N)$   and $\|u_0(x) \|_{ L^{\frac{N(l-1)}3}}$ is sufficiently small, and the  nonlinear functions   $\varphi(u)=O(1)|u |^{p}$ and $\psi(u)=O(1)|u |^l$ as $u\rightarrow 0$, where $p=\frac{2l+1}3$,
 the author proved that there exists a unique global  smooth solution $u \in L^{\infty}\left(0,\infty;L^{\frac{N(l-1)}3}(\mathbb{R}^N)\right)$ for   problem (\ref{1-1}). Moreover, Liu and Liu \cite{LL} studied the Cauchy problem of the degenerate convective Cahn-Hilliard qeuation
 \begin{equation} \label{1-2}
\left\{ \begin{aligned}
         &\partial_tu+\Delta^2_{x'}u=\Delta_{x'}\varphi(u)-\vec{r}\cdot\nabla\psi(u) , \quad x'\in\mathbb{R}^{N-1}\times(0,\infty),\\
                 &u(x,0)=u_0(x),
                          \end{aligned} \right.
                          \end{equation}
where $\Delta_{x'}=\sum_{i=2}^n\partial_{x_i}^2$ denotes the $x'$ direction Laplacian operator with respect to $x'=(x_2,x_3,\cdots,x_n)$,  $\varphi(u)=O(|u|^{\theta+1})$ and $\psi(u)=O(|u|^{\theta+1})$ with the same growth property and $\theta\geq1$ is an integer. By using  the long-short wave method and the frequency decomposition method, the authors proved the existence of the unique global classical solution with small initial data and discussed the decay estimates.

\begin{remark}
There are also some papers studied the global well-posedness of solutions for  Cauchy problem of the Cahn-Hilliard equation (see e.g.,  Bricmont, Kupiainen and Taskinen \cite{10}, Caffarelli and Muler\cite{Caffarelli}, Liu, Wang and   Zhao\cite{ZHJ}, Cholewa and Rodriguez-Bernal \cite{Cholewa1}, Duan and Zhao \cite{DZ} and the reference cited therein).
\end{remark}

It is worth pointing out that the assumptions imposed on the nonlinear functions $\varphi(u)$ and $\psi(u)$ in \cite{AMPA,LL,10,Caffarelli,ZHJ,Cholewa1,DZ} are too strict. One of the most nature assumption on the nonlinear function $\varphi(u)$ is $\varphi(u)=u^3-u$, which is a double-well potential (the other is logarithmic potential). Moreover, we assume that $\psi(u)=\frac12u^2$, which can be found in \cite{Eur,Eden1,Eden2,Kor} and the reference therein. Thus a natural question is how to prove that the  Cauchy problem (\ref{1-1}) with $\varphi(u)=u^3-u$ and $\psi(u)=\frac12u^2$
admits a unique global smooth solution $u(x,t )$ and how to get the optimal temporal decay estimates? The main purpose of
our present paper is devoted to the above problems.   That is, we will consider the global existence and   decay rate of solutions for the Cauchy problem of convective Cahn-Hilliard equation in $\mathbb{R}^3$:
 \begin{equation} \label{1-3}
\left\{ \begin{aligned}
         &\partial_tu+\Delta^2u=\Delta(u^3-u)+u\cdot\nabla u,   \quad x\in\mathbb{R}^3,~t>0,\\
                 &u(\cdot,0)=u_0(\cdot),   \quad x\in\mathbb{R}^3,
                          \end{aligned} \right.
                          \end{equation}
which is equivalent to the following form:
 \begin{equation} \label{3-0}
\left\{ \begin{aligned}
         &\partial_tu+\Delta^2u-\Delta u=\Delta(u^3-2u)+u\cdot\nabla u,   \quad x\in\mathbb{R}^3,~t>0,\\
                 &u(\cdot,0)=u_0(\cdot),   \quad x\in\mathbb{R}^3.
                          \end{aligned} \right.
                          \end{equation}
\begin{remark}In this paper, $\nabla^l$ with an integral $l\geq0$ stands for the usual spatial derivatives of order $l$. If $l<0$ or $l$ is not a positive integer, $\nabla^l$ stands for $\Lambda^l$. We also use $\dot{H}^s(\mathbb{R}^3)$ ($s\in\mathbb{R}$) to denote the homoegneous Sobolev spaces on $\mathbb{R}^3$ with the norm $\|\cdot\|_{H^s}$ defined by $\|f\|_{H^s}:=\|\Lambda^sf\|_{L^2}$, and we use $H^s(\mathbb{R}^3)$ and $L^p(\mathbb{R}^3)$ ($1\leq p\leq\infty$) to describe the usual Sobolev spaces with the norm $\|\cdot\|_{H^s}$ and the usual $L^p$ space with the norm $\|\cdot\|_{L^p}$.
\end{remark}

First of all, by using Banach fixed point theorem, we consider the local well-posedness of solutions to the Cauchy problem (\ref{3-0}) in $\mathbb{R}^3$. More precisely, we prove
the following theorem:
\begin{theorem}[Local well-posedness]
\label{thm1.0}
Suppose that $u_0\in H^2(\mathbb{R}^3) $. Then, there exists a small time $T>0$ and a unique strong solution $u(x,t)$ to system (\ref{3-0}) satisfying
\begin{equation} \label{local-1}
 u\in L^{\infty}([0,\tilde{T}];H^2)\bigcap L^2(0,\tilde{T};H^4).
                          \end{equation}
\end{theorem}

The second purpose of this paper is to prove some global well-posedness results for the Cauchy problem (\ref{3-0}) in $\mathbb{R}^3$. For $N\geq 1$, define
$$
\mathcal{E}_N(t)=\sum_{l=0}^N\|\nabla^lu\|_{L^2}^2,
$$
and the corresponding dissipation rate with minimum derivative counts by
$$
\mathcal{D}_N(t)=\sum_{l=0}^N(\|\nabla^{l+1}u\|_{L^2}^2+\|\nabla^{l+2}u\|_{L^2}^2).
$$
Our   result on the global well-posedness of solutions of Cauchy problem (\ref{1-3}) is stated in the following theorem.
\begin{theorem}
\label{thm1.1}
Let $N\geq 1$, suppose that the initial data $u_0\in H^N(\mathbb{R}^3)$, and there exists a constant   $\delta_0>0$ such that if \begin{equation}\label{bound}\mathcal{E}_1(0)\leq\delta_0,
\end{equation}
then there exists a unique global solution $u(x,t)$ satisfying that for all $t\geq0$,
\begin{equation}
\label{1-5}
\sup_{0\leq t\leq\infty}\mathcal{E}_1(t)+\int_0^{\infty}\mathcal{D}_1(s)ds\leq C\mathcal{E}_1(0).
\end{equation}
Moreover, if $N\geq 2$, then for all $t>0$, the following inequality holds:
\begin{equation}
\label{1-6}
\sup_{0\leq t\leq\infty}\mathcal{E}_N(t)+\int_0^{\infty}\mathcal{D}_N(s)ds\leq C\mathcal{E}_N(0).
\end{equation}
\end{theorem}
 
 The temporal decay rate of solutions is also an interesting topic in the study of  dissipative equations. One of the main tools to study the temporal decay rate is
Fourier splitting method, which was introduced by Schonbek in \cite{S1,S2}. Laterly, this method was well extended to investigate the decay for the solutions of PDE from mathematical physics. In \cite{ZHJ}, by using Fourier splitting method, Liu, Wang and Zhao studied the temporal decay rate of the solution, and its derivatives for the Cauchy problem of Cahn-Hilliard equation  with $\varphi(u)=O(|u|^p)$ for some $p>0$. In this paper, we   improve Liu, Wang and Zhao's results, assume that $\varphi(u)$ is a double-well potential,  study the decay rate of global solutions for problem (\ref{1-3}). More precisely, we establish the following result:
\begin{theorem}
\label{thm1.2}
Suppose that all assumptions in Theorem \ref{thm1.1} hold. Let $u(x,t)$ be the solution to the problem (\ref{1-3}) constructed in Theorem \ref{thm1.1}. Moreover, assume $u_0\in L^{p}(\mathbb{R}^3)$ $(\frac32\leq p\leq 2)$, then the following decay estimate holds:
\begin{equation}
\label{1-9}
\|\nabla^ku(t)\|_{H^{N-k}}\leq C(1+t)^{-\sigma_k},\quad\hbox{for}~k=0,1,\cdots, N-1,
\end{equation}where $$\sigma_k=\frac32\left(\frac1p -\frac12\right)+\frac k2 .
$$\end{theorem}

The rest of this paper is organized as follows. First of all, in Section 2, we give some useful results and lemmas which will be used in our proofs. Then, in Section 3, we prove theorem \ref{thm1.0} on the local well-posedness of solutions for Cauchy problem (\ref{3-0}).   Section 4 is devoted to prove theorem \ref{thm1.1} on the small initial data global well-posedness of solutions. In Section 5, we derive the evolution of the negative Sobolev norms of the solution and establish the decay estimates of problem (\ref{3-0}).

\section{Preliminaries}
In this section, we introduce some helpful results in $\mathbb{R}^3$.

The following Gagliardo-Nirenberg inequality was proved in \cite{Nirenberg}.
\begin{lemma}[\cite{Nirenberg}]
\label{lem2.1}
Let $0\leq m,\alpha\leq l$, then we have
\begin{equation}\label{2-1}
\|\nabla^{\alpha}f\|_{L^p}\lesssim\|\nabla^mf\|_{L^q}^{1-\theta}\|\nabla^lf\|_{L^r}^{\theta},
\end{equation}
where $\theta\in[0,1]$ and $\alpha$ satisfies
\begin{equation}\label{2-2}
\frac{\alpha}3-\frac1p=\left(\frac m3-\frac1q\right)(1-\theta)+\left(\frac l3-\frac1r\right)\theta.
\end{equation}
Here, when $p=\infty$, we require that $0<\theta<1$.
\end{lemma}

We also introduce the Hardy-Littlewood-Sobolev theorem, which implies the following $L^p$ type inequality.
\begin{lemma}[\cite{Stein,15}]
\label{lem2.3}
Let $0\leq s<\frac32$, $1<p\leq 2$ and $\frac12+\frac s3=\frac1p$, then
\begin{equation}
\label{2-4}
\|f\|_{\dot{H}^{-s}}\lesssim\|f\|_{L^p}.
\end{equation}
\end{lemma}

The following special Sobolev interpolation lemma will be used in the proof of Theorem \ref{thm1.2}.
\begin{lemma}[\cite{T,W,Stein}]
\label{lem2.2}
Let $s,k\geq0$ and $l\geq0$, then
\begin{equation}
\label{2-3}
\|\nabla^lf\|_{L^2}\leq\|\nabla^{l+k}f\|_{L^2}^{ 1-\theta }\|f\|_{\dot{H}^{-s}}^{ \theta },\quad\hbox{with}~\theta=\frac k{l+k+s}.
\end{equation}
\end{lemma}
\section{Local Well-posedness}

We will   prove the local well-posedness by using Banach fixed point theorem. Let
$$
\mathcal{A}:=\{v\in C([0,T];H^2),~~~\|v\|_{L^{\infty} (0,T;H^2)} \leq R\},
$$
for some positive constant $R$ to be determined latter.

Assume that $\tilde{u} \in\mathcal{A}$ be given   and $ \tilde{u} (\cdot,0)= u_0 $. Consider
\begin{equation} \label{eq-1}
\left\{ \begin{aligned}
        &u_t+\Delta^2u-\Delta u=\Delta[(\tilde{u}-\sqrt{2})(\tilde{u}+\sqrt{2})u]+\tilde{u}\cdot\nabla u, \\
                  &  u(\cdot,0)=u_0,
                          \end{aligned} \right.
                          \end{equation}
Let $u(x,t)$ be the unique strong solution to (\ref{eq-1}). Define the fixed point map $F: \tilde{u}  \in\mathcal{A}  \rightarrow u\in\mathcal{A} $. We will prove that the map $F$ maps $\mathcal{A}  $ into $\mathcal{A} $ for suitable constant $R$ and small $T>0$ and $F$ is a contraction mapping on $\mathcal{A} $. Therefore, $F$ has a unique fixed point in $\mathcal{A} $. This proves the result.

In the following, we  establish some technical lemmas.
\begin{lemma}
\label{lema-1}
Let $\tilde{u} \in\mathcal{A}  $ be given  and $ \tilde{u}(\cdot,0) = u_0 $. Assume that the  constant $ \tilde{C}_0>0$  is independent of $R$. Then, there exists a unique strong solution $u(x,t)$ for system (\ref{eq-1}) such that
\begin{equation}\label{eq-3}
\|u\|_{L^{\infty}(0,T;H^2)}  \leq \tilde{C}_0,
\end{equation}
for some small $T>0$.
\end{lemma}
\begin{proof}
Since system (\ref{eq-1}) is   linear with regular $\tilde{u}$, whose existence and uniqueness   was proved in Temem \cite{Temam}, then we only need to prove the a priori estimates (\ref{eq-3}) in the following.
Multiplying (\ref{eq-1})$_1$  by $u$,  integrating by parts over $\mathbb{R}^3$, we derive that
\begin{equation}
\begin{aligned}\label{eq-4}&
\frac12\frac d{dt} \|u\|^2_{L^2}+\|\Delta u\|_{L^2}^2+\|\nabla u\|^2_{L^2}
\\
=&\int_{\mathbb{R}^3}(\tilde{u}-\sqrt{2})(\tilde{u}+\sqrt{2})u\Delta udx+\int_{\mathbb{R}^3}(\tilde{u}\cdot\nabla u)udx
\\
\leq&C\|\Delta u\|_{L^2}\|u\|_{L^6}\|\tilde{u}-\sqrt{2}\|_{L^6}\|\tilde{u}+\sqrt{2}\|_{L^6}+C\|\tilde{u}\|_{L^3}\|\nabla u\|_{L^2}\|u\|_{L^6}
\\
\leq&C\|\nabla\tilde{u}\|^2_{L^2}\|\nabla u\|_{L^2}\|\Delta u\|_{L^2}+C\|\tilde{u}\|_{L^2}^{\frac12}\|\nabla\tilde{u}\|_{L^2}^{\frac12}\|\nabla u\|^2_{L^2}
\\
\leq&C(\|\nabla\tilde{u}\|^2_{L^2}+\|\nabla\tilde{u}\|_{L^2})(\|\Delta u\|_{L^2}^2+\|\nabla u\|_{L^2}^2)
\\
\leq&C(R^2+R)(\|\Delta u\|_{L^2}^2+\|\nabla u\|_{L^2}^2),
\end{aligned}\end{equation}
which gives
\begin{equation}\label{eq-9}
\|u\|_{L^2}^2+\int_0^T(\|\Delta u\|_{L^2}^2+\|\nabla u\|_{L^2}^2)ds\leq \tilde{C}_0,
\end{equation}
provided with $ (R^2+R)T\leq 1$.
Taking $ \Delta $ to (\ref{eq-1})$_1$, multiplying it by $ \Delta  u$ and integrating over $\mathbb{R}^3$, we see that
\begin{equation}
\begin{aligned}\label{eq-10}&
\frac12\frac d{dt}(\| 
\Delta u\|_{L^2}^2 + \| \Delta^2u\|_{L^2}^2+ \|\nabla\Delta u\|_{L^2}^2
\\
=& \int_{\mathbb{R}^3}\Delta[(\tilde{u}-\sqrt{2})(\tilde{u}+\sqrt{2})u)\Delta^2udx+\int_{\mathbb{R}^3}(\tilde{u}\cdot\nabla u)\cdot\Delta^2udx
\\
\leq&C\|\Delta^2u\|_{L^2}\|\Delta[(\tilde{u}-\sqrt{2})(\tilde{u}+\sqrt{2})u)\|_{L^2}+C\|\Delta^2u\|_{L^2}\|\tilde{u}\cdot\nabla u\|_{L^2}
\\
\leq&C\|\Delta^2u\|_{L^2}(\|\tilde{u}-\sqrt{2} \|_{L^6}\|\tilde{u}+\sqrt{2}\|_{L^6}\|\Delta u\|_{L^6}+\|\Delta(\tilde{u}-\sqrt{2})\|_{L^6}\|\tilde{u}+\sqrt{2}\|_{L^6}\|  u\|_{L^6}\\&+\|\tilde{u}-\sqrt{2} \|_{L^6}\|\Delta(\tilde{u}+\sqrt{2})\|_{L^6}\|  u\|_{L^6})+C\|\Delta^2u\|_{L^2}\|\nabla u\|_{L^6}\|\tilde{u}\|_{L^3}
\\
\leq&C\|\nabla\tilde{u}\|_{L^2}^2\|\Delta^2u\|_{L^2}\|\nabla\Delta u\|_{L^2}+C\|\tilde{u}\|_{L^2}^{\frac12}\|\nabla\tilde{u}\|_{L^2}^{\frac12}\|\Delta^2u\|_{L^2}\|\Delta u\|_{L^2}
\\
\leq&C(R^2+R)(\|\Delta^2u\|_{L^2}^2+\|\nabla\Delta u\|_{L^2}^2)+CR\|\Delta u\|_{L^2}^2,
\end{aligned}\end{equation} 
which leads
\begin{equation}
\label{eq-15}
 \| \Delta u\|^2_{L^2}+ \int_0^T(\| \Delta^2 u\|_{L^2}^2+\|\nabla \Delta  u\|_{L^2}^2)ds\leq \tilde{C}_0,
\end{equation}
as long as $(R^{2}+R^ )T\leq 1$. The proof is complete.

\end{proof}

By using Lemma  \ref{lema-1}, we can take $R=  \sqrt{\tilde{C}_0} $, and thus, $F$ maps $\mathcal{A} $ into $\mathcal{A} $. In the following, we prove that $F$ is a contraction mapping in the sense of weaker norm.
\begin{lemma}
\label{lem-3}
There exists a constant $\delta\in(0,1)$ such that for any $ \tilde{u}_i $ $(i=1,2)$,
\begin{equation}
\label{eq-16}\begin{aligned}
\|F(\tilde{u}_1 )-F(\tilde{u}_2 )\|_{L^2(0,T;H^2)}
\leq \delta
\|\tilde{u}_1-\tilde{u}_2\|_{L^2(0,T;H^2)}  ,\end{aligned}
\end{equation}
for some small $T>0$.
\end{lemma}
\begin{proof}
Suppose that $ u_i(x,t)$ $(i=1,2)$ are the solutions to problem (\ref{eq-1})  corresponding to $  \tilde{u}_i $.
Denote
$$
u=u_1-u_2,\quad~~\tilde{u}=\tilde{u}_1-\tilde{u}_2,
$$
we have
\begin{equation}\label{u2}\begin{aligned}&
u_t+\Delta ^2u-\Delta u\\=&\Delta[(\tilde{u}_1-\sqrt{2})(\tilde{u}_1+\sqrt{2})u+(\tilde{u}_1-\sqrt{2})(\tilde{u} +\sqrt{2})u_2+(\tilde{u}-\sqrt{2})(\tilde{u}_2-\sqrt{2})u_2]\\&+\tilde{u}_1\cdot\nabla u+\tilde{u}\cdot\nabla u_2.\end{aligned}
\end{equation}
Multiplying (\ref{u2}) by $u$ and integrating on the whole space, then, after integration by parts, we get
\begin{equation}
\begin{aligned}\label{eq-17}
&\frac12\frac d{dt} \|u\|_{L^2}^2+\|\Delta u\|^2_{L^2}+\|\nabla u\|_{L^2}^2
\\
=&\int_{\mathbb{R}^3}(\tilde{u}_1-\sqrt{2})(\tilde{u}_1+\sqrt{2})u\Delta udx+\int_{\mathbb{R}^3}(\tilde{u}_1-\sqrt{2})(\tilde{u} +\sqrt{2})u_2\Delta udx\\&+\int_{\mathbb{R}^3}(\tilde{u}-\sqrt{2})(\tilde{u}_2+\sqrt{2})u_2\Delta udx+\int_{\mathbb{R}^3}(\tilde{u}_1\cdot\nabla u)\cdot udx+\int_{\mathbb{R}^3}(\tilde{u}\cdot\nabla u_2)\cdot udx
\\
\leq&\frac12\|\Delta u\|_{L^2}^2+C\|\tilde{u}_1-\sqrt{2}\|^2_{L^{\infty}}\|\tilde{u}_1+\sqrt{2}\|_{L^{\infty}}^2\|u\|^2_{L^2}+C\|\tilde{u}_1-\sqrt{2}\|_{L^6}^2
\|\tilde{u}+\sqrt{2}\|_{L^6}^2\|u_2\|_{L^6}^2\\&+C\|\tilde{u} -\sqrt{2}\|_{L^6}^2\|\tilde{u}_2+\sqrt{2}\|_{L^6}^2\|u_2\|^2_{L^6}+C\|\tilde{u}_1\|_{L^6}\|\nabla u\|_{L^3}\|u\|_{L^2}+C\|\tilde{u}\|_{L^6}\|\nabla u_2\|_{L^3}\|u\|_{L^2}
\\
\leq&\frac12\|\Delta u\|_{L^2}^2+\frac12\|\nabla u\|_{L^2}^2+C\|u\|^2_{L^2}+C\|\nabla\tilde{u}\|_{L^2}^2.
\end{aligned}\end{equation}
Using the Gronwall's inequality, taking $T$ small enough, we arrive at (\ref{eq-9}) and complete the proof.
\end{proof}

Next, we give the proof of Theorem \ref{thm1.0}.
\begin{proof}[Proof of Theorem \ref{thm1.0}]
By Lemmas \ref{lema-1}, \ref{lem-3}   and a variant of the Banach fixed point theorem, using weak compactness, we complete the proof.
\end{proof}

\section{Small initial data global well-posedness}

In this section, on the basis of the assumptions of Theorem \ref{thm1.1}, we establish the energy estimates of the solution to the Cauchy problem (\ref{3-0}).
\begin{lemma}
\label{lem3.1} Assume $T>0$ and $0<\delta\ll 1$. Let
\begin{equation}\label{3-1}
\sup_{0\leq t\leq T}\|u(t)\|_{H^1}\leq\delta,
\end{equation} and all assumptions in Theorem \ref{thm1.1} hold. Then, for any $t\in[0,T]$ and integer $k\geq0$, we have
\begin{equation}
\label{3-2}
\begin{aligned}&
 \frac d{dt}\sum_{l=k}^{k+1}\|\nabla^lu\|_{L^2}^2+\sum_{l=k}^{k+1}\|\nabla^{l+2}u\|_{L^2}^2+\sum_{l=k}^{k+1}\|\nabla^{l+1}u\|_{L^2}^2
\\
\leq&C_l\sum_{l=k}^{k+1}(\|u\|_{H^1}+\|u\|^2_{H^1})(\|\nabla^{l+1}u\|_{L^2}^2+\|\nabla^{l+2}u\|_{L^2}^2).
\end{aligned}
\end{equation}
\end{lemma}
\begin{proof}

For any integer $k\geq0$, applying $\nabla^l$ ($l=k,k+1 $) to (\ref{3-0})$_1$, multiplying the resulting identities by $\nabla^lu$, integrating over $\mathbb{R}^3$ by parts, we find that
\begin{equation}
\label{3-3}
\begin{aligned}&\frac12\frac d{dt}\|\nabla^lu\|_{L^2}^2+\|\nabla^{l+2}u\|_{L^2}^2+\|\nabla^{l+1}u\|_{L^2}^2
\\
=&\int_{\mathbb{R}^3}\nabla^l(u^3-2u)\cdot\nabla^{l+2}udx+\int_{\mathbb{R}^3}\nabla^l(u\cdot\nabla u)\cdot\nabla^ludx
 .
\end{aligned}\end{equation}
Note that
\begin{equation}
\begin{aligned}\label{3-4}&
\int_{\mathbb{R}^3}\nabla^l(u^3-2u)\cdot\nabla^{l+2}udx
\\
\lesssim&\|\nabla^{l+2}u\|_{L^2}\|\nabla^l[u(u+\sqrt{2})(u-\sqrt{2})]\|_{L^2}
\\
\lesssim&\|\nabla^{l+2}u\|_{L^2}\left(\|\nabla^lu\|_{L^6}\|u+\sqrt{2}\|_{L^6}\|u-\sqrt{2}\|_{L^6}\right.\\&\left.+\|\nabla^l(u+\sqrt{2})\|_{L^6}\|u \|_{L^6}\|u-\sqrt{2}\|_{L^6}+\|\nabla^l(u-\sqrt{2})\|_{L^6}\|u+\sqrt{2}\|_{L^6}\|u\|_{L^6}\right)
\\
\lesssim&\|\nabla u\|_{L^2}^2\|\nabla^{l+2}u\|_{L^2}\|\nabla^{l+1}u\|_{L^2}
\\
\lesssim&\|\nabla u\|_{L^2}^2(\|\nabla^{l+2}u\|_{L^2}^2+\|\nabla^{l+1}u\|_{L^2}^2),\end{aligned}
\end{equation}
and
\begin{equation}
\begin{aligned}\label{3-8}
\int_{\mathbb{R}^3}\nabla^l(u\cdot\nabla u)\cdot\nabla^ludx
=&-\frac12\int_{\mathbb{R}^3}\nabla^l(\nabla\cdot u^2)\cdot\nabla^ludx
\\
\lesssim&\|\nabla^l(\nabla\cdot u^2)\|_{L^{\frac65}}\|\nabla^lu\|_{L^6}
\\
\lesssim& \|u\|_{L^3}\|\nabla^{l+1}u\|_{L^2}\|\nabla^iu\|_{L^6}
\\
\lesssim&\|u\|_{L^3}\|\nabla^{l+1}u\|_{L^2}^2.
\end{aligned}
\end{equation}
Plugging (\ref{3-4}) and (\ref{3-8}) into (\ref{3-3}), we conclude that
\begin{equation}
\label{3-14}
\frac12\frac d{dt}\|\nabla^lu\|_{L^2}^2+\|\nabla^{l+2}u\|_{L^2}^2+\|\nabla^{l+1}u\|_{L^2}^2\leq C(\|u\|_{H^1}^2+\|u\|_{H^1})(\|\nabla^{l+2}u\|_{L^2}^2+\|\nabla^{l+1}u\|_{L^2}^2),
\end{equation}
then we complete the proof.
\end{proof}

Now, on the basis of the assumption that $\|u_0\|_{H^1}$ is sufficiently small, we propose to prove the existence and uniqueness of global solution to Cauchy problem (\ref{1-3}).
\begin{proof}[Proof of Theorem \ref{thm1.1}]There are two steps for us to prove Theorem \ref{thm1.1}.

Step 1. Global small $\mathcal{E}_1$ solution.

It follows from the assumption (\ref{3-1}), taking $k=0$ in (\ref{3-2}), we have for any $t\in[0,T]$,
\begin{equation}\begin{aligned}
\label{6-1}&
\frac d{dt}\sum_{l=0}^1\|\nabla^lu\|_{L^2}^2+\sum_{l=0}^{ 1}\|\nabla^{l+2}u\|_{L^2}^2+\sum_{l=0}^{1}\|\nabla^{l+1}u\|_{L^2}^2
\\
\leq&C_2(\sqrt{\mathcal{E}_1(t)}+\mathcal{E}_1(t))\mathcal{D}_1(t)\leq C_2\delta(\delta+1)\mathcal{D}_1(t) .
\end{aligned}
\end{equation}
By (\ref{6-1}), we can choose a sufficiently small $\delta$, such that
\begin{equation}
\label{6-2}
\mathcal{E}_1(t)+\int_0^t\mathcal{D}_1(\tau)d\tau\leq\tilde{C}_2\mathcal{E}_1(0),\quad\forall t\in[0,T].
\end{equation}
Suppose that $\varepsilon_0=\delta+\delta^2$ is a positive constant,
where $\delta>0$ is given in Lemmas \ref{lem3.1}. We also choose initial data $u_0$ and small constant $\delta_0$, such that
$$
\sqrt{\mathcal{E}_1(0)}\leq\sqrt{\delta_0}:=\frac{\varepsilon_0}{2\sqrt{1+\tilde{C}_2}}.
$$
Next, define the lifespan of solutions of problem (\ref{3-0}) by
$$
T:=\sup\left\{t:\sup_{0\leq\tau\leq t}\sqrt{\mathcal{E}_1(s)}\leq\varepsilon_0\right\}.
$$
Note that
$$
\sqrt{\mathcal{E}_1(0)}\leq\frac{\varepsilon_0}{2\sqrt{1+\tilde{C}_2}}\leq\frac{\varepsilon_0}2<\varepsilon_0\leq\varepsilon,
$$
hence $T>0$ holds true from the local existence result  and the continuation argument. If the time $T$ is finite, from the definition of $T$, we have
$$
\sup_{0\leq\tau\leq T}\sqrt{\mathcal{E}_1(\tau)}=\varepsilon_0.
$$
However, on the basis of the uniform a priori estimate (\ref{6-2}), the following inequalities hold:
$$
\sup_{0\leq\tau\leq T}\sqrt{\mathcal{E}_1(\tau)}\leq\sqrt{\tilde{C}_2}\sqrt{\mathcal{E}_1(0)}\leq\frac{\sqrt{\tilde{C}_2}\varepsilon_0}{2\sqrt{1+\tilde{C}_2}}\leq\frac{\varepsilon_0}2,
$$
which is a contradiction. Therefore, $T=\infty$, and the local solution $u(t)$ obtained in Theorem \ref{thm1.0} can be extent to infinite time. Thus, there exists a unique solution $u(t)\in L^{\infty}([0,\infty];H^1)$ for the Cauchy problem (\ref{3-0}), and the inequality (\ref{1-5}) holds.

Recall  (\ref{3-2}), for $N\geq2$, we have
\begin{equation}
\begin{aligned}\label{6-3}& \frac d{dt}\sum_{l=0}^{N}\|\nabla^lu\|_{L^2}^2+\sum_{l=0}^{N}\|\nabla^{l+2}u\|_{L^2}^2+\sum_{l=0}^{N}\|\nabla^{l+1}u\|_{L^2}^2
\\
\leq&(\sqrt{\mathcal{E}_2(t)}+\mathcal{E}_2(t))\mathcal{D}_N(t). \end{aligned}
\end{equation}
By using the smallness of $\varepsilon_0$ and (\ref{6-3}), we deduce that
$$
\mathcal{E}_N(t)+\int_0^t\mathcal{D}_N(t)\leq C\mathcal{E}_N(0),\quad\forall t\in[0,\infty],
$$
this complete the proof of Theorem \ref{thm1.1}.

\end{proof}

\section{Decay estimates}
In this section, we first derive the evolution of the negative Sobolev norms of the solution to the Cauchy problem (\ref{1-3}).
In order to estimate the convective term and the double-well potential, we shall restrict ourselves to that $s\in[0,\frac12]$.

For the homogeneous Sobolev space, the following lemma holds:
\begin{lemma}
\label{lem5.1}
Suppose that all the assumptions in Lemma \ref{lem3.1} are in force. For $s\in[0,\frac12]$, we have
\begin{equation}
\label{5-1}\begin{aligned}
\frac d{dt}\|u(t)\|^2_{\dot{H}^{-s}}+\|\nabla^2u(t)\|^2_{\dot{H}^{-s}}+\|\nabla u(t)\|^2_{\dot{H}^{-s}}\lesssim\|\nabla u\|_{H^1}^2\|u(t)\| _{\dot{H}^{-s}},
\end{aligned}
\end{equation}where the parameter $\delta$ is the same as (\ref{3-1}).
\end{lemma}
\begin{proof}
Applying $\Lambda^{-s}$ to (\ref{3-0}), multiplying the resulting identities by $\Lambda^{-s}u$, and then integrating over $\mathbb{R}^3$ by parts, we deduce that
\begin{equation}
\label{5-2}
\begin{aligned}
&\frac12\frac d{dt}\|\Lambda^{-s}u\|_{L^2}^2+\|\Lambda^{-s}\nabla^2u\|_{L^2}^2+\|\Lambda^{-s}\nabla u\|_{L^2}^2
\\
=&\int_{\mathbb{R}^3}\Lambda^{-s}(u\cdot\nabla u)\cdot\Lambda^{-s}udx+\int_{\mathbb{R}^3}\Lambda^{-s}\Delta(u^3-2u)\cdot\Lambda^{-s}udx.
\end{aligned}
\end{equation}
For the first term of the right hand side of (\ref{5-2}), we have
\begin{equation}
\label{5-3}
\begin{aligned}
\int_{\mathbb{R}^3}\Lambda^{-s}(u\cdot\nabla u)\cdot\Lambda^{-s}udx
\leq&\|\Lambda^{-s}(u\cdot\nabla u)\|_{L^2}\|\Lambda^{-s}u\|_{L^2}
\\
\lesssim&\|u\cdot\nabla u\|_{L^{\frac1{\frac12+\frac s3}}}\|\Lambda^{-s}u\|_{L^2}
\\
\lesssim&\|u\|_{L^{\frac 3s}}\|\nabla u\|_{L^2}\|\Lambda^{-s}u\|_{L^2}
\\
\lesssim&\|\nabla u\|_{L^2}^{\frac 12+s}\|\nabla^2u\|_{L^2}^{\frac12-s}\|\nabla u\|_{L^2}\|\Lambda^{-s}u\|_{L^2}
\\
\lesssim&\|\Lambda^{-s}u\|_{L^2}(\|\nabla u\|_{L^2}^2+\|\nabla^2u\|_{L^2}^2)
.
\end{aligned}
\end{equation}
For the second term of the right hand side of  of (\ref{5-2}), we have
\begin{equation}
\begin{aligned}
\label{5-4}
&\int_{\mathbb{R}^3}\Lambda^{-s}\Delta(u^3-2u)\cdot\Lambda^{-s}udx
\\
=&\int_{\mathbb{R}^3}\Lambda^{-s}\Delta[u(u+\sqrt{2})(u-\sqrt{2})]\cdot\Lambda^{-s}udx
\\
\leq&\|\Lambda^{-s}u\|_{L^2}\|\lambda^{-s}\Delta[u(u+\sqrt{2})(u-\sqrt{2})]\|_{L^2}
\\
\lesssim&\|\Lambda^{-s}u\|_{L^2}\left[\|\Lambda^{-s}(u(u+\sqrt{2})\Lambda^2(u-\sqrt{2}))\|_{L^2}+\|\Lambda^{-s}(u(u-\sqrt{2})\Lambda^2(u+\sqrt{2}))\|_{L^2}
\right.\\
&+
\|\Lambda^{-s}((u-\sqrt{2})(u+\sqrt{2})\Lambda^2u)\|_{L^2}+\|\Lambda^{-s}( \nabla u\cdot\nabla(u-\sqrt{2})\cdot(u+\sqrt{2}))\|_{L^2}
\\
&\left.+
\|\Lambda^{-s}( \nabla u\cdot\nabla(u+\sqrt{2})\cdot(u-\sqrt{2}))\|_{L^2}+\|\Lambda^{-s}( \nabla (u+\sqrt{2})\cdot\nabla(u-\sqrt{2})\cdot u))\|_{L^2}\right]
\\
\lesssim&\|\Lambda^{-s}u\|_{L^2}\left[\| u(u+\sqrt{2})\Lambda^2(u-\sqrt{2}) \|_{L^{\frac1{\frac12+\frac s3}}}+\| u(u-\sqrt{2})\Lambda^2(u+\sqrt{2}) \|_{L^{\frac1{\frac12+\frac s3}}}
\right.\\
&+
\| (u-\sqrt{2})(u+\sqrt{2})\Lambda^2u \|_{L^{\frac1{\frac12+\frac s3}}}+\| |\nabla u||\nabla(u-\sqrt{2})||u+\sqrt{2}| \|_{L^{\frac1{\frac12+\frac s3}}}
\\
&\left.+
\| |\nabla u||\nabla(u+\sqrt{2})||u-\sqrt{2}|\|_{L^{\frac1{\frac12+\frac s3}}}+ |\nabla (u+\sqrt{2})||\nabla(u-\sqrt{2})||u | \|_{L^{\frac1{\frac12+\frac s3}}}\right]
\\
\lesssim&\|\Lambda^{-s}u\|_{L^2}(\|u\|_{L^{\infty}}\|u+\sqrt{2}\|_{L^{\frac3s}}\|\nabla^2(u-\sqrt{2})\|_{L^2}+
\|u\|_{L^{\infty}}\|u-\sqrt{2}\|_{L^{\frac3s}}\|\nabla^2(u+\sqrt{2})\|_{L^2}
\\
&+\|u-\sqrt{2}\|_{L^{\infty}}\|u+\sqrt{2}\|_{L^{\frac3s}}\|\nabla^2u\|_{L^2}+
\|\nabla u\|_{L^3}\|\nabla(u+\sqrt{2})\|_{L^6}\|u-\sqrt{2}\|_{L^{\frac3s}}\\
&+\|\nabla u\|_{L^3}\|\nabla(u-\sqrt{2})\|_{L^6}\|u+\sqrt{2}\|_{L^{\frac3s}}
+\|\nabla (u-\sqrt{2})\|_{L^3}\|\nabla(u+\sqrt{2})\|_{L^6}\|u \|_{L^{\frac3s}})
\\
\lesssim&\|\nabla u\|_{H^1}^3\|\Lambda^{-s}u\|_{L^2}\lesssim
 \delta\|\nabla u\|_{H^1}^2\|\Lambda^{-s}u\|_{L^2},
\end{aligned}
\end{equation}
where we have used
$$
\|v\|_{L^{\infty}}\lesssim \|\nabla v\|_{L^2}^{\frac12}\|\Delta v\|_{L^2}^{\frac12},
$$
$$
\|v\|_{L^3}\lesssim\|v\|_{L^2}^{\frac12}\|\nabla v\|_{L^2}^{\frac12},
$$
and
$$
\|v\|_{L^{\frac 3s}}\lesssim\|\nabla v\|_{L^2}^{\frac12+s}\|\Delta v\|_{L^2}^{\frac12-s}.
$$
Plugging the estimates (\ref{5-3}) and (\ref{5-4}) into (\ref{5-2}), we deduce (\ref{5-1}). Hence, the proof is complete.
\end{proof}

In the following, we devoted to establish the temporary decay rate of unique global solutions for Cauchy problem (\ref{3-0}). On the basis of the conclusions of Theorem \ref{thm1.1} and Lemma \ref{lem5.1}, we proceed to prove this result.
\begin{proof}[Proof of Theorem \ref{thm1.2}]
Define
$$
\mathcal{E}_{-s}(t):=\|\Lambda^{-s}u(t)\|_{L^2}^2.
$$
Then, integrating in time (\ref{5-1}) of Lemma \ref{lem5.1}, by the bound (\ref{bound}), we obtain that for $s\in[0,\frac12]$,
\begin{equation}
\label{6-4}
\begin{aligned}
\mathcal{E}_{-s}(t)\leq&\mathcal{E}_{-s}(0)+C\int_0^t\|\nabla u\|_{H^1}^2\sqrt{\mathcal{E}_{-s}(\tau)}d\tau
\\
\leq&C_0\left(1+\sup_{0\leq\tau\leq t}\sqrt{\mathcal{E}_{-s}(\tau)}d\tau\right),
\end{aligned}
\end{equation}
which implies  
\begin{equation}
\label{6-5}
\|\Lambda^{-s}u(t)\|_{L^2}^2\leq C_0,\quad\forall s\in[0,\frac12].
\end{equation}
Moreover, if  $k=1,2,\cdots,N-2$, we may use Lemma \ref{lem2.2} to have
$$
\|\nabla^{k+1}f\|_{L^2}\geq C\|\Lambda^{-k}f\|_{L^2}^{-\frac1{k+s}}\|\nabla^kf\|_{L^2}^{1+\frac1{k+s}}.
$$
Then, by this fact and (\ref{6-5}), we get
\begin{equation}\label{7-1}
\|\nabla^{k+1}u\|_{L^2}^2\geq C_0(\|\nabla^{k}u\|_{L^2}^2)^{1+\frac1{k+s}}.
\end{equation}
On the other hand, we may define a family of energy functions and the corresponding dissipation rates with minimum derivatives counts as
\begin{equation}
\label{6-6}
\mathcal{E}_k^{k+1}:=\sum_{l=k}^{k+1}\|\nabla^lu(t)\|_{L^2}^2,
\end{equation}
and
\begin{equation}
\label{6-7}
\mathcal{D}_k^{k+1}:=\sum_{l=k}^{k+1}(\|\nabla^l\nabla u\|_{L^2}^2+\|\nabla^{l+2}u\|_{L^2}^2).
\end{equation}
Taking into account Lemma \ref{lem3.1} and Theorem \ref{thm1.1}, we have that for $k=0,1,\cdots,N-2$ that
\begin{equation}
\label{6-8}
\frac d{dt}\mathcal{E}_k^{k+1}+\mathcal{D}_k^{k+1}\leq0.
\end{equation}
Note that
\begin{equation}
\label{6-9}
\mathcal{D}_k^{k+1}\geq\sum_{l=k}^{k+1}\|\nabla^{l+2}u\|_{L^2}^2.
\end{equation}
Combining (\ref{7-1}) and (\ref{6-9}) together gives
\begin{equation}
\label{6-10}
\mathcal{D}_k^{k+1}\gtrsim \left(\mathcal{E}_k^{k+1}\right)^{1+\frac1{k+s}}.
\end{equation}
From (\ref{6-8}) and (\ref{6-10}), we conclude that
\begin{equation}
\label{6-11}
\frac d{dt}\mathcal{E}_k^{k+1}+\left(\mathcal{E}_k^{k+1}\right)^{1+\frac1{k+s}}\leq 0,
\end{equation}
with $k=0,1,\cdots,N-2$. Solving (\ref{6-11})  directly gives
\begin{equation}
\label{6-11x}\mathcal{E}_k^{k+1}\leq C_0(1+t)^{-k-s},\quad\hbox{for}~k=1,2,\cdots,N-2.
\end{equation}
Note that the Hardy-Littlewood-Sobolev theorem implies that for $p\in(1,2]$,  $L^{p}(\mathbb{R}^3)\subset\dot{H}^{-s}(\mathbb{R}^3)$ with $s=3(\frac1p-\frac12)\in[0,\frac32)$. Therefore, based on   (\ref{6-11x}), we   obtain
\begin{equation*}\label{xx-1}\begin{aligned}&
\|\nabla^lu\|_{H^{N-l}} \leq C(1+t)^{-\left[\frac32\left(\frac1p -\frac12\right)+\frac k2\right]},\quad\hbox{for}~l=0,1,\cdots, N-1.\end{aligned}
\end{equation*}
Then, the inequality
(\ref{1-9}) holds and  we complete the proof of Theorem \ref{thm1.2}.

\section*{Acknowledgement}
This paper was supported by the Fundamental Research Funds for the Central Universities (grant No. N2005031).

\end{proof}

}
\end{document}